\documentclass{article}

\usepackage{amssymb}
\usepackage{amsthm}
\usepackage{amsmath}
\usepackage{mathtools}
\usepackage{enumitem}
\usepackage[capitalize]{cleveref}
\usepackage[all]{xy}
\usepackage{tikz-cd}

\numberwithin{equation}{section}

\theoremstyle{definition}
\newtheorem{definition}[equation]{Definition}
\theoremstyle{plain}

\newtheorem{proposition}[equation]{Proposition}
\newtheorem{theorem}[equation]{Theorem}
\newtheorem{corollary}[equation]{Corollary}
\newtheorem{claim}[equation]{Claim}
\theoremstyle{remark}
\newtheorem{remark}[equation]{Remark}

\crefname{equation}{}{}
\crefname{enumi}{}{}

\newcommand{\todo}[1]{}
\newcommand{\abs}[1]{\lvert{#1}\rvert}
\newcommand{\id}{{\rm id}}
\newcommand{\pr}{\mathrm{pr}}
\newcommand{\ev}{\mathrm{ev}}
\newcommand{\res}{\mathrm{res}}
\renewcommand{\deg}[1]{\mathopen|#1\mathclose|}
\newcommand{\K}{\mathbb K}
\newcommand{\Q}{\mathbb Q}
\newcommand{\Z}{\mathbb Z}
\newcommand{\C}{\mathbb C}
\newcommand{\cochain}[2][*]{C^{#1}(#2)}
\newcommand{\cohom}[2][*]{H^{#1}(#2)}
\newcommand{\homol}[2][*]{H_{#1}(#2)}
\newcommand{\homolshift}[2][*]{{\mathbb H}_{#1}(#2)}
\newcommand{\im}{\operatorname{Im}}

\renewcommand{\hom}{\mathrm{Hom}}
\newcommand{\ext}{\mathrm{Ext}}
\newcommand{\tor}{\mathrm{Tor}}
\renewcommand{\lim}{{\varprojlim}}
\newcommand{\limone}{{\varprojlim}^1}

\newcommand{\shriekhomol}[1]{{#1}^!}
\newcommand{\shriek}[1]{{#1}_!}
\newcommand{\incl}{\mathrm{incl}}
\newcommand{\comp}{\mathrm{comp}}
\newcommand{\diag}{\Delta}
\newcommand{\inclconst}{c}
\newcommand{\shriekinclconst}{\shriek\gamma}
\newcommand{\shriekdiag}{\shriek\delta}
\newcommand{\qis}{\simeq}
\newcommand{\tpow}[1]{^{\otimes #1}}
\newcommand{\susp}[2][]{s^{#1}#2}
\newcommand{\EM}{\operatorname{EM}}

\newcommand{\bprod}{\mu}
\newcommand{\bcop}{\delta}
\newcommand{\dprod}{\mu^\vee}
\newcommand{\dcop}{\delta^\vee}
\DeclareMathOperator{\map}{Map}
\newcommand{\loopsp}[1]{L#1}
\newcommand{\spheresp}[2][]{S^{#1}#2}

\newcommand{\orirev}{\tau}

\title{Coproducts in brane topology}
\author{Shun Wakatsuki}
\date{}

\begin{document}
\maketitle
\begin{abstract}
  % In string topology,
  % the loop product and coproduct on the free loop space of Gorenstein spaces are studied.
  We extend the loop product and the loop coproduct
  to the mapping space from the $k$-dimensional sphere, or more generally from any $k$-manifold,
  to a $k$-connected space with finite dimensional rational homotopy group, $k\geq 1$.
  The key to extending the loop coproduct is the fact that
  the embedding $M\rightarrow M^{S^{k-1}}$ is of ``finite codimension''
  in a sense of Gorenstein spaces.
  Moreover, we prove the associativity, commutativity, and Frobenius compatibility of them.
\end{abstract}

\section{Introduction}
Chas and Sullivan \cite{chas-sullivan} introduced the loop product
on the homology $\homol{\loopsp{M}}$ of the free loop space $\loopsp{M}=\map(S^1, M)$ of a manifold.
Cohen and Godin \cite{cohen-godin} extended this product to other string operations,
including the loop coproduct.

Generalizing these constructions,
F\'elix and Thomas \cite{felix-thomas09} defined the loop product and coproduct
in the case $M$ is a Gorenstein space.
A Gorenstein space is a generalization of a manifold in the point of view of Poincar\'e duality,
including the classifying space of a connected Lie group
and the Borel construction of a connected oriented closed manifold and a connected Lie group.
But these operations tend to be trivial in many cases.
Let $\K$ be a field of characteristic zero.
For example,
the loop coproduct is trivial for a manifold with the Euler characteristic zero \cite[Corollary 3.2]{tamanoi},
the composition of the loop coproduct followed by the loop product is trivial for any manifold \cite[Theorem A]{tamanoi},
and the loop product over $\K$ is trivial for the classifying space of a connected Lie group \cite[Theorem 14]{felix-thomas09}.
A space with the nontrivial composition of loop coproduct and product is not found.

On the other hand,
Sullivan and Voronov
\todo{article([SV05] in \cite{cohen-hess-voronov}) not found}%
generalized the loop product to the sphere space $\spheresp[k]{M}=\map(S^k, M)$\todo{Voronov's notation} for $k\geq 1$.
This product is called the \textit{brane product}.
See \cite[Part I, Chapter 5]{cohen-hess-voronov}.

In this article,
we will generalize the loop coproduct to sphere spaces,
to construct nontrivial and interesting operations.
We call this coproduct the \textit{brane coproduct}.

Here, we review briefly the construction of the loop product and the brane product.
For simplicity, we assume $M$ is a connected oriented closed manifold of dimension $m$.
The loop product is constructed as a mixture of
the Pontrjagin product $\homol{\Omega M \times \Omega M} \rightarrow \homol{\Omega M}$
defined by the composition of based loops
and the intersection product $\homol{M\times M} \rightarrow \homol[*-m]{M}$.
More precisely, we use the following diagram
\begin{equation}
  \label{equation:loopProdDiagram}
  \xymatrix{
    \loopsp{M}\times\loopsp{M} \ar_{\ev_1\times\ev_1}[d] & \loopsp{M}\times_M\loopsp{M} \ar_-\incl[l]\ar[d]\ar^-\comp[r] & \loopsp{M}\\
    M\times M & M. \ar_-\diag[l]
  }
\end{equation}
Here,
the square is a pullback diagram by
the diagonal map $\diag$ and
the evaluation map $\ev_1$ at $1$, identifying $S^1$ with the unit circle $\{z\in \C\mid \abs{z}=1\}$,
and $\comp$ is the map defined by the composition of loops.
Since the diagonal map $\diag\colon M\rightarrow M\times M$ is an embedding of finite codimension,
we have the shriek map $\shriekhomol{\diag}\colon \homol{M\times M}\rightarrow\homol[*-m]{M}$,
which is called the intersection product.
Using the pullback diagram, we can ``lift'' $\shriekhomol{\diag}$ to
$\shriekhomol{\incl}\colon \homol{\loopsp{M}\times \loopsp{M}}\rightarrow\homol[*-m]{\loopsp{M}\times_M \loopsp{M}}$.
Then, we define the loop product to be the composition
$\comp_*\circ\shriekhomol{\incl}\colon\homol{\loopsp{M}\times \loopsp{M}}\rightarrow\homol[*-m]{\loopsp{M}}.$

The brane product can be defined by a similar way.
Let $k$ be a positive integer.
We use the diagram
\begin{equation*}
  \xymatrix{
    \spheresp[k]{M}\times\spheresp[k]{M} \ar[d] & \spheresp[k]{M}\times_M\spheresp[k]{M} \ar_\incl[l]\ar[d]\ar^-\comp[r] & \spheresp[k]{M}\\
    M\times M & M. \ar_-\diag[l]
  }
\end{equation*}
Since the base map of the pullback diagram is the diagonal map $\diag$,
which is the same as that for the loop product,
we can use the same method to define the shriek map
$\shriekhomol{\incl}\colon \homol{\spheresp[k]{M}\times \spheresp[k]{M}}\rightarrow\homol[*-m]{\spheresp[k]{M}\times_M \spheresp[k]{M}}$.
Hence we define the brane product to be the composition
$\comp_*\circ\shriekhomol{\incl}\colon\homol{\spheresp[k]{M}\times \spheresp[k]{M}}\rightarrow\homol[*-m]{\spheresp[k]{M}}.$

Next, we review the loop coproduct.
Using the diagram
\begin{equation}
  \label{equation:loopCopDiagram}
  \xymatrix{
    \loopsp{M} \ar_{\ev_1\times\ev_{-1}}[d] & \loopsp{M}\times_M\loopsp{M} \ar_-\comp[l]\ar[d]\ar^-\incl[r] & \loopsp{M}\times\loopsp{M}\\
    M\times M & M, \ar_-\diag[l]
  }
\end{equation}
we define the loop coproduct to be the composition
$\incl_*\circ\shriekhomol{\comp}\colon \homol{\loopsp{M}}\rightarrow\homol[*-m]{\loopsp{M}\times\loopsp{M}}$.

But the brane coproduct cannot be defined in this way.
To construct the brane coproduct,
we have to use the diagram
\begin{equation*}
  \xymatrix{
    \spheresp[k]{M} \ar_\res[d] & \spheresp[k]{M}\times_M\spheresp[k]{M} \ar_-\comp[l]\ar[d]\ar^-\incl[r] & \spheresp[k]{M}\times\spheresp[k]{M}\\
    \spheresp[k-1]{M} & M. \ar_-\inclconst[l]
  }
\end{equation*}
Here, $\inclconst\colon M\rightarrow\spheresp[k-1]{M}$ is the embedding by constant maps
and $\res\colon\spheresp[k]{M}\rightarrow\spheresp[k-1]{M}$ is the restriction map to $S^{k-1}$,
which is embedded to $S^k$ as the equator.
In a usual sense, the base map $\inclconst$ is not an embedding of finite codimension.
But using the algebraic method of F\'elix and Thomas \cite{felix-thomas09},
we can consider this map as an embedding of codimension $\bar{m}=\dim\Omega M$,
which is the dimension as a $\K$-Gorenstein space and is finite when $\pi_*(M)\otimes\K$ is of finite dimension.
Hence, under this assumption, we have the shriek map
$\shriekhomol{\inclconst}\colon\homol{\spheresp[k-1]{M}}\rightarrow\homol[*-\bar{m}]{M}$
and the lift
$\shriekhomol{\comp}\colon\homol{\spheresp[k]{M}}\rightarrow\homol[*-\bar{m}]{\spheresp[k]{M}\times_M\spheresp[k]{M}}$.
This enables us to define the brane coproduct to be the composition
$\incl_*\circ\shriekhomol{\comp}\homol{\spheresp[k]{M}}\rightarrow\homol[*-\bar{m}]{\spheresp[k]{M}\times\spheresp[k]{M}}$.

% Note that we construct the brane coproduct
% as a special case of a more general operation defined by connected sums.
% See \cref{section:definitionOfBraneOperations} for details.
More generally,
using connected sums,
we define the product and coproduct for mapping spaces from manifolds.
Let $S$ and $T$ be manifolds of dimension $k$.
Let $M$ be a $k$-connected $\K$-Gorenstein space of finite type.\todo{$k$-conn?}
Denote $m = \dim M$.
Then we define the \textit{$(S,T)$-brane product}
\begin{equation*}
  \bprod_{ST}\colon \homol{M^S\times M^T} \rightarrow \homol[*-m]{M^{S\#T}}
\end{equation*}
using the diagram
\begin{equation}
  \label{equation:STProdDiagram}
  \xymatrix{
    M^S\times M^T \ar[d] & M^S\times_MM^T \ar_-\incl[l]\ar[d]\ar^-\comp[r] & M^{S\#T}\\
    M\times M & M. \ar_-\diag[l]
  }
\end{equation}
Assume that $M$ is $k$-connected and
$\pi_*(M)\otimes\K = \bigoplus_n\pi_n(M)\otimes\K$ is of finite dimension.
Then the iterated based loop space $\Omega^{k-1} M$ is a Gorenstein space,
and denote $\bar{m} = \dim\Omega^{k-1} M$.
Then we define the \textit{$(S,T)$-brane coproduct}
\begin{equation*}
  \bcop_{ST}\colon \homol{M^{S\#T}} \rightarrow \homol[*-\bar{m}]{M^S\times M^T}
\end{equation*}
using the diagram
\begin{equation}
  \label{equation:STCopDiagram}
  \xymatrix{
    M^{S\#T}\ar[d] & M^S\times_MM^T \ar_-\comp[l]\ar[d]\ar^-\incl[r] & M^S\times M^T\\
    \spheresp[k-1]{M} & M. \ar_-\inclconst[l]
  }
\end{equation}
Note that, if we take $S=T=S^{k}$,
then $\bprod_{ST}$ and $\bcop_{ST}$ are the brane product and coproduct, respectively.

Next, we study some fundamental properties of the brane product and coproduct.
For the loop product and coproduct on Gorenstein spaces,
Naito \cite{naito13} showed their associativity and the Frobenius compatibility.
In this article,
we generalize them to the case of the brane product and coproduct.
Moreover, we show the commutativity of the brane product and coproduct,
which was not known even for the case of the loop product and coproduct on Gorenstein spaces.

\begin{theorem}
  \label{theorem:associativeFrobenius}
  Let $M$ be a $k$-connected space with $\dim\pi_*(M)\otimes\K < \infty$.
  % \todo{finiteness only for (2) (3)}
  Then the above product and coproduct satisfy following properties.
  % Then we have
  \begin{enumerate}[label={\rm(\arabic*)}]
    \item \label{item:assocProd} The product is associative and commutative.
    \item \label{item:assocCop} The coproduct is associative and commutative.
    \item \label{item:Frob} The product and coproduct satisfy the Frobenius compatibility.
  \end{enumerate}
  In particular,
  if we take $S=T=S^k$,
  the shifted homology
  $\homolshift{\spheresp[k]{M}} = \homol[*+m]{\spheresp[k]{M}}$
  is a non-unital and non-counital Frobenius algebra,
  where $m$ is the dimension of $M$ as a Gorenstein space.
\end{theorem}
\todo{Explain Frobenius algebra}
\todo{commutativity,unitality}

Note that $M$ is a Gorenstein space by the assumption $\dim\pi_*(M)\otimes\K < \infty$
(see \cref{proposition:FinDimImplyGorenstein}).
The associativity of the product holds even
if we assume that $M$ is a Gorenstein space instead of assuming $\dim\pi_*(M)\otimes\K < \infty$.
But we need the assumption to prove the commutativity of the \textit{product}.

A non-unital and non-counital Frobenius algebra corresponds to a ``positive boundary'' TQFT,
in the sense that TQFT operations are defined
only when each component of the cobordism surfaces has
a \textit{positive} number of incoming and outgoing boundary components \cite{cohen-godin}.

See \cref{section:proofOfAssocAndFrob} for the precise statement and the proof of the associativity, the commutativity and the Frobenius compatibility.
It is interesting that
the proof of the commutativity of the loop coproduct (i.e., $k=1$) is easier
than that of the brane coproduct with $k\geq 2$.
In fact,
we prove the commutativity of the loop coproduct
using the explicit description of the loop coproduct constructed in \cite{wakatsuki16}.
On the other hand,
we prove the commutativity of the brane coproduct with $k\geq 2$ directly from the definition.

Moreover, we compute an example of the brane product and coproduct.
Here, we consider the shifted homology $\homolshift[*]{\spheresp[k]{M}} = \homol[*+m]{\spheresp[k]{M}}$.
We also have the shifts of the brane product and coproduct on $\homolshift{\spheresp[k]{M}}$
with the sign determined by the Koszul sign convention.
\begin{theorem}
  \label{theorem:braneOperationsOfSphere}
  \newcommand{\M}{S^{2n+1}}
  % The brane product and coproduct are both nontrivial for $M=S^{2n+1}$.
  % Moreover, the brane coproduct
  % $\homol{\spheresp{S^{2n+1}}}\rightarrow\homol{\spheresp{S^{2n+1}}\times\spheresp{S^{2n+1}}}$
  % is injective.
  The shifted homology $\homolshift{\spheresp[2]\M}$, $n\geq 1$, equipped with the brane product $\bprod$
  is isomorphic to the exterior algebra
  $\wedge(y,z)$ with $\deg{y}=-2n-1$ and $\deg{z}=2n-1$.
  The brane coproduct $\bcop$ is described as follows.
  \begin{align*}
    \bcop(1) &= 1\otimes yz - y\otimes z + z\otimes y + yz\otimes 1\\
    \bcop(y) &= y\otimes yz + yz\otimes y\\
    \bcop(z) &= z\otimes yz + yz\otimes z\\
    \bcop(yz) &= - yz\otimes yz
  \end{align*}
\end{theorem}

Note that both of the brane product and coproduct are non-trivial.
Moreover, $(\bcop\otimes 1)\circ\bcop\neq 0$
in contrast with the case of the loop coproduct,
in which the similar composition is always trivial \cite[Theorem A]{tamanoi}.

On the other hand,
the brane coproduct is trivial in some cases.

\begin{theorem}
  \label{theorem:coprodTrivForPure}
  If the minimal Sullivan model $(\wedge V, d)$ of $M$ is pure
  and satisfies $\dim V^{\mathrm{even}}>0$,
  then the brane coproduct on $\homol{\spheresp[2]{M}}$ is trivial.
\end{theorem}

See \cref{definition:pureSullivanAlgebra} for the definition of a pure Sullivan algebra.

\begin{remark}
  \label{remark:generalizedConnectedSum}
  \todo{orientation?}
  If we fix embeddings of disks $D^k\hookrightarrow S$ and $D^k\hookrightarrow T$
  instead of assuming $S$ and $T$ being manifolds,
  we can define the product and coproduct using ``connected sums'' defined by these embedded disks.
  Moreover, if we have two disjoint embeddings $i,j\colon D^k\hookrightarrow S$ to the same space $S$,
  we can define the ``connected sum'' along $i$ and $j$,
  and hence we can define the product and coproduct using this.
  We call these $(S,i,j)$-brane product and coproduct,
  and give definitions in \cref{section:definitionOfSijBraneOperations}.
\end{remark}

\todo{outline}
\cref{section:constructionFelixThomas} contains brief background material on string topology on Gorenstein spaces.
We define the $(S,T)$-brane product and coproduct in \cref{section:definitionOfBraneOperations}
and $(S,i,j)$-brane product and coproduct in \cref{section:definitionOfSijBraneOperations}.
Here, we defer the proof of \cref{corollary:extSphereSpace}
to  \cref{section:constructModel}.
In \cref{section:computeExample},
we compute examples and prove \cref{theorem:braneOperationsOfSphere} and \cref{theorem:coprodTrivForPure}.
\cref{section:proofOfAssocAndFrob} is devoted to the proof of \cref{theorem:associativeFrobenius},
where we defer the determination of some signs
to \cref{section:determineSign} and \cref{section:proofOfExtAlgebraic}.

\tableofcontents

\section{Construction by F\'elix and Thomas}
\label{section:constructionFelixThomas}
In this section,
we recall the construction of the loop product and coproduct by F\'elix and Thomas \cite{felix-thomas09}.
Since the cochain models are good for fibrations,
the duals of the loop product and coproduct are defined at first,
and then we define the loop product and coproduct as the duals of them.
Moreover we focus on the case when the characteristic of the coefficient $\K$ is zero.
So we make full use of rational homotopy theory.
For the basic definitions and theorems on homological algebra and rational homotopy theory,
we refer the reader to \cite{felix-halperin-thomas01}.

\begin{definition}
  [{\cite{felix-halperin-thomas88}}]
  Let $m\in\Z$ be an integer.
  \begin{enumerate}
    \item An augmented dga (differential graded algebra) $(A,d)$ is called a \textit{($\K$-)Gorenstein algebra of dimension} $m$ if
      \begin{equation*}
        \dim \ext_A^l(\K, A) =
        \begin{cases}
          1 & \mbox{ (if $l = m$)} \\
          0 & \mbox{ (otherwise),}
        \end{cases}
      \end{equation*}
      where the field $\K$ and the dga $(A,d)$ are $(A,d)$-modules via the augmentation map and the identity map, respectively.
    \item A path-connected topological space $M$ is called a \textit{($\K$-)Gorenstein space of dimension} $m$
      if the singular cochain algebra $\cochain{M}$ of $M$ is a Gorenstein algebra of dimension $m$.
  \end{enumerate}
\end{definition}

Here, $\ext_A(M, N)$ is defined using a semifree resolution of $(M,d)$ over $(A,d)$,
for a dga $(A,d)$ and $(A,d)$-modules $(M,d)$ and $(N,d)$.
$\tor_A(M,N)$ is defined similarly.
See \cite[Section 1]{felix-thomas09} for details of semifree resolutions.

An important example of a Gorenstein space is given by the following \lcnamecref{proposition:FinDimImplyGorenstein}.

\begin{proposition}
  [{\cite[Proposition 3.4]{felix-halperin-thomas88}}]
  \label{proposition:FinDimImplyGorenstein}
  A 1-connected topological space $M$ is a $\K$-Gorenstein space if $\pi_*(M)\otimes\K$ is finite dimensional.
  Similarly, a Sullivan algebra $(\wedge V, d)$ is a Gorenstein algebra if $V$ is finite dimensional.
\end{proposition}

Note that this \lcnamecref{proposition:FinDimImplyGorenstein} is stated
only for $\Q$-Gorenstein spaces in \cite{felix-halperin-thomas88},
but the proof can be applied for any $\K$ and Sullivan algebras.

Let $M$ be a 1-connected $\K$-Gorenstein space of dimension $m$ whose cohomology $\cohom{M}$ is of finite type.
As a preparation to define the loop product and coproduct, F\'elix and Thomas proved the following theorem.

\begin{theorem}
  [{\cite[Theorem 12]{felix-thomas09}}]
  \label{theorem:ExtDiagonal}
  The diagonal map $\diag\colon M \rightarrow M^2$ makes $\cochain{M}$ into a $\cochain{M^2}$-module.
  We have an isomorphism
  \[
  \ext_{\cochain{M^2}}^*(\cochain{M}, \cochain{M^2}) \cong \cohom[*-m]{M}.
  \]
  % Let $n$ be a positive integer.
  % The diagonal map $\diag\colon M \rightarrow M^n$ makes $\cochain{M}$ into a $\cochain{M^n}$-module.
  % Then we have an isomorphism
  % \[
  % \ext_{\cochain{M^n}}^*(\cochain{M}, \cochain{M^n}) \cong \cohom[*-(n-1)m]{M}.
  % \]
\end{theorem}

By \cref{theorem:ExtDiagonal}, we have $\ext_{\cochain{M^2}}^m(\cochain{M}, \cochain{M^2})\cong\cohom[0]{M}\cong\K$, hence the generator
\[
\shriek\diag \in \ext_{\cochain{M^2}}^m(\cochain{M}, \cochain{M^2})
\]
is well-defined up to the multiplication by a non-zero scalar.
We call this element the \textit{shriek map} for $\diag$.

Using the map $\shriek\diag$, we can define the duals of the loop product and coproduct.
Then, using the diagram \cref{equation:loopProdDiagram},
we define the dual of the loop product to be the composition
\begin{equation*}
  \shriek\incl\circ\comp^*\colon\cohom{\loopsp{M}}
  \xrightarrow{\comp^*}\cohom{\loopsp{M}\times_M\loopsp{M}}
  \xrightarrow{\shriek\incl}\cohom[*+m]{\loopsp{M}\times\loopsp{M}}.
\end{equation*}
Here, the map $\shriek\incl$ is defined by the composition
\begin{equation*}
  \begin{array}{l}
    \cohom{\loopsp{M}\times_M \loopsp{M}}
    \xleftarrow[\cong]{\mathrm{EM}} \tor^*_{\cochain{M^2}}(\cochain{M},\cochain{\loopsp{M}\times \loopsp{M}}) \\
    \xrightarrow{\tor_\id(\shriek\diag, \id)}\tor^{*+m}_{\cochain{M^2}}(\cochain{M^2},\cochain{\loopsp{M}\times \loopsp{M}})
    \xrightarrow[\cong]{} \cohom[*+m]{\loopsp{M}\times \loopsp{M}},
  \end{array}
\end{equation*}
where the map $\mathrm{EM}$ is the Eilenberg-Moore map, which is an isomorphism (see \cite[Theorem 7.5]{felix-halperin-thomas01} for details).
Similarly, using the diagram \cref{equation:loopCopDiagram},
we define the dual of the loop coproduct to be the composition
\begin{equation*}
  \shriek\comp\circ\incl^*\colon\cohom{\loopsp{M}\times\loopsp{M}}
  \xrightarrow{\incl^*}\cohom{\loopsp{M}\times_M\loopsp{M}}
  \xrightarrow{\shriek\comp}\cohom{\loopsp{M}}.
\end{equation*}
Here, the map $\shriek\comp$ is defined by the composition
\begin{equation*}
  \begin{array}{l}
    \cohom{\loopsp{M}\times_M \loopsp{M}}
    \xleftarrow[\cong]{\mathrm{EM}} \tor^*_{\cochain{M^2}}(\cochain{M}, \cochain{\loopsp{M}})\\
    \xrightarrow{\tor_\id(\shriek\diag, \id)} \tor^{*+m}_{\cochain{M^2}}(\cochain{M^2}, \cochain{\loopsp{M}})
    \xrightarrow[\cong]{} \cohom[*+m]{\loopsp{M}}.
  \end{array}
\end{equation*}

\section{Definition of $(S,T)$-brane coproduct}
\label{section:definitionOfBraneOperations}
Let $\K$ be a field of characteristic zero,
$S$ and $T$ manifolds of dimension $k$,\todo{connected? connected sum at where}
and $M$ a $k$-connected Gorenstein space of finite type.
% Let $M$ be a 1-connected space with $\pi_*(M)\otimes\K$ is of finite dimension.
As in the construction by F\'elix and Thomas,
which we reviewed in \cref{section:constructionFelixThomas},
we construct the duals
\begin{align*}
  \dprod_{ST}\colon& \cohom{M^{S\#T}} \rightarrow \cohom[*+\dim M]{M^S\times M^T}\\
  \dcop_{ST}\colon& \cohom{M^S\times M^T} \rightarrow \cohom[*+\dim\Omega^{k-1} M]{M^{S\#T}}
\end{align*}
of the $(S,T)$-brane product and the $(S,T)$-brane coproduct.

The $(S,T)$-brane product is defined by a similar way to that of F\'elix and Thomas.
Using the diagram \cref{equation:STProdDiagram},
we define $\dprod_{ST}$ to be the composition
\begin{equation*}
  \shriek\incl\circ\comp^*\colon\cohom{M^{S\#T}}
  \xrightarrow{\comp^*}\cohom{M^S\times_MM^T}
  \xrightarrow{\shriek\incl}\cohom[*+m]{M^S\times M^T}.
\end{equation*}
Here, the map $\shriek\incl$ is defined by the composition
\begin{equation*}
  \begin{array}{l}
    \cohom{M^S\times_M M^T}
    \xleftarrow[\cong]{\mathrm{EM}} \tor^*_{\cochain{M^2}}(\cochain{M},\cochain{M^S\times M^T}) \\
    \xrightarrow{\tor_\id(\shriek\diag, \id)}\tor^{*+m}_{\cochain{M^2}}(\cochain{M^2},\cochain{M^S\times M^T})
    \xrightarrow[\cong]{} \cohom[*+m]{M^S\times M^T},
  \end{array}
\end{equation*}

Next,\todo{,?} we begin the definition of the $(S,T)$-brane coproduct.
But \cref{theorem:ExtDiagonal} cannot be applied to this case
since the base map of the pullback is $\inclconst\colon M\rightarrow\spheresp[k-1]{M}$.

Instead of \cref{theorem:ExtDiagonal},
we use the following theorem to define the $(S,T)$-brane coproduct.
A graded algebra $A$ is \textit{connected} if $A^0=\K$ and $A^i=0$ for any $i<0$.
A dga $(A,d)$ is \textit{connected} if $A$ is connected.

\newcommand{\dimb}{\bar{m}}
\begin{theorem}
  \label{theorem:extAlgebraic}
  \todo{characteristic can be nonzero}
  \todo{Gor dim $m$? $\bar{m}$?}
  % Let $(A,d)$ be a connected commutative dga of finite type
  % and $B$ a connected commutative graded algebra.
  % Let $(A\otimes B, d)$ be a dga such that
  % the canonical inclusion from $(A,d)$ is a dga homomorphism
  % and $(A\otimes B, d)$ is semifree over $(A,d)$ by the inclusion.
  Let $(A\otimes B, d)$ be a dga such that
  $A$ and $B$ are connected commutative graded algebras,
  $(A,d)$ is a sub dga of finite type,
  and $(A\otimes B, d)$ is semifree over $(A,d)$.
  Let $\eta\colon(A\otimes B, d) \rightarrow (A,d)$ be a dga homomorphism.
  Assume that the following conditions hold.
  \begin{enumerate}[label={\rm(\alph{enumi})}]
    \item \label{item:assumpResId} The restriction of $\eta$ to $A$ is the identity map of $A$.
    \item \label{item:assumpGorenstein} The dga $(B,\bar{d})=\K\otimes_A(A\otimes B, d)$ is a Gorenstein algebra of dimension $\dimb$.
    \item \label{item:assump1conn} For any $b\in B$, the element $db-\bar{d}b$ lies in $A^{\geq 2}\otimes B$.
  \end{enumerate}
  Then we have an isomorphism
  \begin{equation*}
    \ext^*_{A\otimes B}(A, A\otimes B) \cong \cohom[*-\dimb]{A}.
  \end{equation*}
\end{theorem}

This can be proved by a similar method to \cref{theorem:ExtDiagonal} \cite[Theorem 12]{felix-thomas09}.
The proof is given in \cref{section:proofOfExtAlgebraic}.

Applying to sphere spaces,
we have the following corollary.
\begin{corollary}
  \label{corollary:extSphereSpace}
  Let $M$ be a $(k-1)$-connected (and 1-connected) space with $\pi_*(M)\otimes\K$ of finite dimension.
  Then we have an isomorphism
  \begin{equation*}
    \ext^*_{\cochain{\spheresp[k-1]{M}}}(\cochain{M}, \cochain{\spheresp[k-1]{M}}) \cong \cohom[*-\bar{m}]{M},
  \end{equation*}
  where $\bar{m}$ is the dimension of $\Omega^{k-1}M$ as a Gorenstein space.
\end{corollary}

To prove the corollary,
we need to construct models of sphere spaces satisfying the conditions of \cref{theorem:extAlgebraic}.
This will be done in \cref{section:constructModel}.

Note that, since $\spheresp[0]M = M\times M$, this is a generalization of \cref{theorem:ExtDiagonal}
(in the case that the characteristic of $\K$ is zero).

Assume that $M$ is a $k$-connected space with $\pi_*(M)\otimes\K$ of finite dimension.
Then we have
$\ext^{\bar{m}}_{\cochain{\spheresp[k-1]{M}}}(\cochain{M}, \cochain{\spheresp[k-1]{M}}) \cong \cohom[0]{M}\cong\K$,
hence the shriek map for $\inclconst\colon M\rightarrow\spheresp[k-1]{M}$ is defined to be the generator
\begin{equation*}
  \shriek\inclconst\in
  \ext^{\bar{m}}_{\cochain{\spheresp[k-1]{M}}}(\cochain{M}, \cochain{\spheresp[k-1]{M}}),
\end{equation*}
which is well-defined up to the multiplication by a non-zero scalar.
Using $\shriek\inclconst$ with the diagram \cref{equation:STCopDiagram},
we define the dual $\dcop_{ST}$ of the $(S,T)$-brane coproduct to be the composition
\begin{equation*}
  \shriek\comp\circ\incl^*\colon\cohom{M^S\times M^T}
  \xrightarrow{\incl^*}\cohom{M^S\times_MM^T}
  \xrightarrow{\shriek\comp}\cohom{M^{S\#T}}.
\end{equation*}
Here, the map $\shriek\comp$ is defined by the composition
\begin{equation*}
  \begin{array}{l}
    \cohom{M^S\times_M M^T}
    \xleftarrow[\cong]{\mathrm{EM}} \tor^*_{\cochain{\spheresp[k-1]{M}}}(\cochain{M}, \cochain{M^{S\#T}})\\
    \xrightarrow{\tor_\id(\shriek\inclconst, \id)} \tor^{*+\bar{m}}_{\cochain{\spheresp[k-1]{M}}}(\cochain{\spheresp[k-1]{M}}, \cochain{M^{S\#T}})
    \xrightarrow[\cong]{} \cohom[*+\bar{m}]{\loopsp{M}}.
  \end{array}
\end{equation*}
Note that the Eilenberg-Moore isomorphism can be applied since $\spheresp[k-1]{M}$ is 1-connected.

\section{Definition of $(S,i,j)$-brane product and coproduct}
\label{section:definitionOfSijBraneOperations}
\newcommand{\connsum}{\#}
\newcommand{\wedgesum}{\bigvee}
\newcommand{\diskcrush}{Q}
\newcommand{\smalld}{D}
\newcommand{\intd}{\smalld^\circ}
\newcommand{\bd}{\partial \smalld}

In this section, we give a definition of $(S,i,j)$-brane product and coproduct.
Let $S$ be a topological space, and $i$ and $j$ embeddings $D^k\rightarrow S$.
Fix a small $k$-disk $\smalld\subset D^k$ and denote its interior by $\intd$ and its boundary by $\bd$.
Then we define three spaces $\connsum(S,i,j)$, $\diskcrush(S,i,j)$, and $\wedgesum(S,i,j)$ as follows.
The space $\connsum(S,i,j)$ is obtained from $S\setminus(i(\intd)\cup j(\intd))$
by gluing $i(\bd)$ and $j(\bd)$ by an orientation reversing homeomorphism.
We obtain $\diskcrush(S,i,j)$ by collapsing two disks $i(\smalld)$ and $j(\smalld)$ to two points, respectively.
$\wedgesum(S,i,j)$ is defined as the quotient space of $\diskcrush(S,i,j)$ identifying the two points.
Then, since the quotient space $D^k/\smalld$ is homeomorphic to the disk $D^k$,
we identify $\diskcrush(S,i,j)$ with $S$ itself.
By the above definitions, we have the maps
$\connsum(S,i,j)\rightarrow\wedgesum(S,i,j)$ and $S=\diskcrush(S,i,j)\rightarrow\wedgesum(S,i,j)$.
For a space $M$, these maps induce the maps
$\comp\colon M^{\wedgesum(S,i,j)}\rightarrow M^{\connsum(S,i,j)}$
and $\incl\colon M^{\wedgesum(S,i,j)}\rightarrow M^S$.
Moreover, we have diagrams
\begin{equation*}
  \xymatrix{
    M^S \ar[d] & M^{\wedgesum(S,i,j)} \ar_-\incl[l]\ar[d]\ar^-\comp[r] & M^{\connsum(S,i,j)}\\
    M\times M & M \ar_-\diag[l]
  }
\end{equation*}
and
\begin{equation*}
  \xymatrix{
    M^{\connsum(S,i,j)}\ar[d] & M^{\wedgesum(S,i,j)} \ar_-\comp[l]\ar[d]\ar^-\incl[r] & M^S\\
    \spheresp[k-1]{M} & M \ar_-\inclconst[l]
  }
\end{equation*}
in which the squares are pullback diagrams.
\todo{remove label?}
If $M$ is a $k$-connected space with $\pi_*(M)\otimes\K$ of finite dimension,
we define the $(S,i,j)$-brane product and coproduct by a similar method to \cref{section:definitionOfBraneOperations},
using these diagrams instead of the diagrams \cref{equation:STProdDiagram} and \cref{equation:STCopDiagram}.
Note that this generalizes $(S,T)$-brane product and coproduct defined in \cref{section:definitionOfBraneOperations}.

\section{Construction of models and proof of \cref{corollary:extSphereSpace}}
\label{section:constructModel}
% Let $M$ be a $(k-1)$-connected (and 1-connected) space with $\pi_*(M)\otimes\K$ of finite dimension.
In this section, we give a proof of \cref{corollary:extSphereSpace},
constructing a Sullivan model of the dga homomorphism
$\inclconst^*\colon \cochain{\spheresp[k-1]{M}}\rightarrow\cochain{M}$
satisfying the assumptions of \cref{theorem:extAlgebraic}.
% Let $(\wedge V, d)$ be a Sullivan model of $M$ such that
% $V^{\leq k-1}=0$ and $V$ is of finite type.
% If $k=1$,
% then we can apply \cref{theorem:extAlgebraic} for the product map
% $(\wedge V, d)^{\otimes 2}\rightarrow (\wedge V, d)$
% (and this is a result due to F\'elix and Thomas in this case).
% We will construct models for $k\geq 2$.

\newcommand{\susplow}{s^{(k-1)}}
\newcommand{\susphigh}{s^{(k)}}
\newcommand{\diffsphere}[1][k-1]{\bar{d}^{(#1)}}
\newcommand{\diffdisk}[1][k]{d^{(#1)}}

First, we construct models algebraically.
Let $(\wedge V, d)$ be a Sullivan algebra.
For an integer $l\in \Z$,
let $\susp[l]V$ be a graded module defined by $(\susp[l]V)^n=V^{n+l}$
and $\susp[l]v$ denotes the element in $\susp[l]V$ corresponding to the element $v \in V$.

Define two derivations $\susplow$ and $\diffsphere$
on the graded algebra $\wedge V\otimes \wedge \susp[k-1]V$ by
% $\susplow(v)=\susp[k-1]v$, $\susplow(\susp[k-1]v)=0$,
% $\diffsphere(v)=dv$, and $\diffsphere(\susp[k-1]v)=(-1)^{k-1}\susplow dv$.
\begin{align*}
  &\susplow(v)=\susp[k-1]v,\quad \susplow(\susp[k-1]v)=0, \\
  &\diffsphere(v)=dv,\ \mbox{and}\quad\diffsphere(\susp[k-1]v)=(-1)^{k-1}\susplow dv.
\end{align*}
Then it is easy to see that
$\diffsphere\circ\diffsphere=0$ and hence
$(\wedge V\otimes \wedge\susp[k-1]V, \diffsphere)$ is a dga.

Similarly, define derivations $\susphigh$ and $\diffdisk$
on the graded algebra $\wedge V\otimes \wedge\susp[k-1]V \otimes\wedge\susp[k]V$ by
% $\susphigh(v)=\susp[k]v$, $\susphigh(\susp[k-1]v)=\susphigh(\susp[k]v)=0$,
% $\diffdisk(v)=dv$, $\diffdisk(\susp[k-1]v)=\diffsphere(\susp[k-1]v)$,
% and $\diffdisk(\susp[k]v)=\susp[k-1]v+(-1)^k\susphigh dv$.
\begin{align*}
  &\susphigh(v)=\susp[k]v,\quad\susphigh(\susp[k-1]v)=\susphigh(\susp[k]v)=0,\quad%
    \diffdisk(v)=dv, \\
  &\diffdisk(\susp[k-1]v)=\diffsphere(\susp[k-1]v),\ %
    \mbox{and}\quad\diffdisk(\susp[k]v)=\susp[k-1]v+(-1)^k\susphigh dv.
\end{align*}
Then it is easy to see that
$\diffdisk\circ\diffdisk=0$ and hence
$(\wedge V\otimes \wedge\susp[k-1]V \otimes\wedge\susp[k]V, \diffdisk)$ is a dga.

Note that the tensor product
$(\wedge V, d) \otimes_{\wedge V\otimes\wedge\susp[k-1]V}(\wedge V\otimes\wedge\susp[k-1]V\otimes\wedge\susp[k]V,\diffdisk)$
is canonically isomorphic to $(\wedge V\otimes\wedge\susp[k]V,\diffsphere[k])$,
where $(\wedge V, d)$ is a $(\wedge V\otimes\wedge\susp[k-1]V, \diffsphere)$-module by the dga homomorphism
$\phi\colon (\wedge V\otimes\wedge\susp[k-1]V, \diffsphere) \rightarrow (\wedge V, d)$
defined by $\phi(v)=v$ and $\phi(\susp[k-1]v)=0$.

It is clear that,
if $V^{\leq k-1}=0$,
the dga $(\wedge V\otimes\wedge\susp[k-1]V,\diffsphere)$ is a Sullivan algebra
and, if $V^{\leq k}=0$,
the dga $(\wedge V\otimes \wedge\susp[k-1]V \otimes\wedge\susp[k]V, \diffdisk)$
is a relative Sullivan algebra over $(\wedge V\otimes\wedge\susp[k-1]V,\diffsphere)$.

\newcommand{\tveps}{\tilde{\varepsilon}}
Define a dga homomorphism
\begin{equation*}
  \tilde{\varepsilon}\colon
  (\wedge V\otimes \wedge\susp[k-1]V \otimes\wedge\susp[k]V, \diffdisk)
  \rightarrow (\wedge V, d)
\end{equation*}
by $\tveps(v)=v$ and $\tveps(\susp[k-1]v)=\tveps(\susp[k]v)=0$.
Then the linear part
\begin{equation*}
  Q(\tveps)\colon
  (V \oplus\susp[k-1]V \oplus\susp[k]V, \diffdisk_0)
  \rightarrow (V, d_0)
\end{equation*}
is a quasi-isomorphism,
and hence $\tveps$ is a quasi-isomorphism \cite[Proposition 14.13]{felix-halperin-thomas01}.

Using these algebras,
we have the following proposition.

\begin{proposition}
  \label{proposition:modelOfInclConst}
  Let $k\geq 2$ be an integer,
  $M$ a $(k-1)$-connected (and 1-connected) Gorenstein space of finite type,
  and $(\wedge V, d)$ its Sullivan model such that $V^{\leq k-1}=0$ and $V$ is of finite type.
  Then the dga homomorphism
  $\phi\colon (\wedge V\otimes\wedge\susp[k-1]V, \diffsphere) \rightarrow (\wedge V, d)$
  is a Sullivan representative of the map
  $\inclconst\colon M\rightarrow\spheresp[k-1]{M}$,
  i.e., there is a homotopy commutative diagram
  \begin{equation*}
    \xymatrix{
      (\wedge V\otimes\wedge\susp[k-1]V, \diffsphere) \ar[r]^-\phi \ar[d]^\qis & (\wedge V, d) \ar[d]^\qis\\
      \cochain{\spheresp[k-1]{M}} \ar[r]^-{\inclconst^*} & \cochain{M}
    }
  \end{equation*}
  such that the vertical arrows are quasi-isomorphisms.
\end{proposition}
\begin{proof}
  We prove the proposition by induction on $k\geq 2$.
  The case $k=2$ is well-known.
  See \cite[Section 15 (c) Example 1]{felix-halperin-thomas01} or \cite[Appendix A]{wakatsuki16} for details.

  Assume that the proposition holds for some $k$.
  Consider the commutative diagram
  \begin{equation*}
    \begin{tikzcd}[row sep=2.5em]
      M \arrow[rr,"="] \arrow[dr,swap,"\inclconst"] \arrow[dd,swap,"="] &&
      M \arrow[dd,swap,"=" near start] \arrow[dr,""] \\
      & \spheresp[k]{M} \arrow[rr,crossing over,"" near start] &&
      M^{D^k} \arrow[dd,"\res"] \\
      M \arrow[rr,"=" near end] \arrow[dr,swap,"="] && M \arrow[dr,swap,"\inclconst"] \\
      & M \arrow[rr,"\inclconst"] \arrow[uu,<-,crossing over,"\ev" near end]&& \spheresp[k-1]{M},
    \end{tikzcd}
  \end{equation*}
  where the front and back squares are pullback diagrams.
  % Since tensor products with relative Sullivan algebras are models of pullbacks of fibrations \cite[Section 15 (c)]{felix-halperin-thomas01},
  Since any pullback diagram of a fibration is modeled by a tensor product of Sullivan algebras \cite[Section 15 (c)]{felix-halperin-thomas01},
  this proves the proposition.
  % \begin{equation*}
  %   \xymatrix{
  %   \spheresp[k]{M} \ar[r] \ar[d] & M^{D^k} \ar[d] \\
  %   M \ar[r] & \spheresp[k-1]{M},
  % }
  % \end{equation*}
\end{proof}

\begin{proof}
  [Proof of \cref{corollary:extSphereSpace}]
  In the case $k=1$, apply \cref{theorem:extAlgebraic} to the product map
  $(\wedge V, d)^{\otimes 2}\rightarrow (\wedge V, d)$.
  (Note that this case is a result of F\'elix and Thomas \cite{felix-thomas09}.)

  In the case $k\geq 2$, using \cref{proposition:modelOfInclConst},
  apply \cref{theorem:extAlgebraic} to the map $\phi$.
\end{proof}

\section{Computation of examples}
\label{section:computeExample}
\newcommand{\M}{S^{2n+1}}

\makeatletter
\newcommand{\@tensorpower}[1]{\ifx#1\relax\else^{\otimes #1}\fi}
\newcommand{\@modelcommand}[3]{%
  \ifx#2\relax
    {#1}\@tensorpower{#3}
  \else
    #1(#2)\@tensorpower{#3}
  \fi
}
\newcommand{\@mpath}{{\mathcal M}_\mathrm{P}}
\newcommand{\@mloop}{{\mathcal M}_\mathrm{L}}
\newcommand{\mpath}[1][\relax]{\@modelcommand{\@mpath}{\relax}{#1}}
\newcommand{\mpathv}[2][\relax]{\@modelcommand{\@mpath}{\wedge #2}{#1}}
\newcommand{\mloop}[1][\relax]{\@modelcommand{\@mloop}{\relax}{#1}}
\newcommand{\mloopv}[2][\relax]{\@modelcommand{\@mloop}{\wedge #2}{#1}}
\newcommand{\mdisk}[1]{\@modelcommand{{\mathcal M}_{D^{#1}}}{\relax}{\relax}}
\newcommand{\msphere}[1]{\@modelcommand{{\mathcal M}_{S^{#1}}}{\relax}{\relax}}
\makeatother

In this section,
we will compute the brane product and coproduct for some examples,
which proves \cref{theorem:braneOperationsOfSphere,theorem:coprodTrivForPure}.

In \cite{naito13}, the duals of the loop product and coproduct are described in terms of Sullivan models using the torsion functor description of \cite{kuribayashi-menichi-naito}.
By a similar method,
we can describe the brane product and coproduct as follows.

Let $M$ be a $k$-connected $\K$-Gorenstein space of finite type
and $(\wedge V, d)$ its Sullivan model such that $V^{\leq k}=0$ and $V$ is of finite type.
Denote $(\wedge V\otimes\wedge\susp[k]{V}, \diffsphere[k])$ by $\msphere{k}$
and $(\wedge V\otimes\wedge\susp[k-1]{V}\otimes\wedge\susp[k]{V}, \diffdisk[k])$ by $\mdisk{k}$
(see \cref{section:constructModel} for the definitions).
Define a relative Sullivan algebra $\mpath=(\wedge V\tpow2\otimes\wedge \susp{V}, d)$
over $(\wedge V, d)\tpow2$ by the formula
\begin{equation*}
  d(\susp v)=1\otimes v - v\otimes 1 - \sum_{i=1}^\infty\frac{(sd)^i}{i!}(v\otimes 1)
\end{equation*}
inductively
(see \cite[Section 15 (c)]{felix-halperin-thomas01} or \cite[Appendix A]{wakatsuki16} for details\todo{wakatsuki16?,inductively?}).
Note that $\cohom{\msphere{k}}\cong\cohom{\spheresp[k]{M}}$ and
$\cohom{\mdisk{k}}\cong\cohom{M^{D^k}}\cong\cohom{M}$.
Then the dual of the brane product on $\cohom{\spheresp[k]{M}}$ is induced by the composition
\begin{equation*}
  \begin{array}{l}
    \msphere{k}
    \xrightarrow{\cong} \wedge V \otimes_{\msphere{k-1}} \mdisk{k}
    \xleftarrow[\qis]{\tilde{\varepsilon}\otimes \id} \mdisk{k} \otimes_{\msphere{k-1}} \mdisk{k}
    \xrightarrow{(\phi\otimes\id)\otimes_\phi(\phi\otimes\id)} \msphere{k} \otimes_{\wedge V} \msphere{k}\\
    \xrightarrow{\cong} \wedge V \otimes_{\wedge V\tpow2}\msphere{k}\tpow2
    \xleftarrow[\qis]{\bar{\varepsilon}\otimes\id} \mpath\otimes_{\wedge V\tpow2}\msphere{k}\tpow2
    \xrightarrow{\shriekdiag\otimes\id} \wedge V\tpow2\otimes_{\wedge V\tpow2}\msphere{k}\tpow2
    \xrightarrow{\cong} \msphere{k}\tpow2,
  \end{array}
\end{equation*}
where $\shriekdiag$ is a representative of $\shriek\diag$.
See \cref{section:constructModel} for the definitions of the other maps.

Assume that $\pi_*(M)\otimes\K$ is of finite dimension.
Then the dual of the brane coproduct is induced by the composition
\begin{equation*}
  \begin{array}{l}
    \msphere{k}\tpow2
    \xrightarrow{\cong} \wedge V\tpow2 \otimes_{\msphere{k-1}\tpow2} \mdisk{k}\tpow2
    \xrightarrow{\mu\otimes_{\mu'}\eta} \wedge V \otimes_{\msphere{k-1}} (\mdisk{k}\otimes_{\msphere{k-1}}\mdisk{k})\\
    \xleftarrow[\qis]{\tilde{\varepsilon}\otimes\id} \mdisk{k}\otimes_{\msphere{k-1}}(\mdisk{k}\otimes_{\msphere{k-1}}\mdisk{k})
    \xrightarrow{\shriekinclconst\otimes\id} \msphere{k-1}\otimes_{\msphere{k-1}}(\mdisk{k}\otimes_{\msphere{k-1}}\mdisk{k})\\
    \xrightarrow{\cong} \mdisk{k}\otimes_{\msphere{k-1}}\mdisk{k}
    \xrightarrow[\qis]{\tilde{\varepsilon}\otimes\id} \wedge V\otimes_{\msphere{k-1}}\mdisk{k}
    \xrightarrow{\cong} \msphere{k},
  \end{array}
\end{equation*}
where $\shriekinclconst$ is a representative of $\shriek\inclconst$,
the maps $\mu$ and $\mu'$ are the product maps,
and $\eta$ is the quotient map.

As a preparation of computation,
recall the definition of a pure Sullivan algebra.

\begin{definition}
  [c.f. {\cite[Section 32]{felix-halperin-thomas01}}]
  \label{definition:pureSullivanAlgebra}
  A Sullivan algebra $(\wedge V, d)$ with $\dim V < \infty$ is called {\it pure}
  if $d(V^{\rm even})=0$ and $d(V^{\rm odd}) \subset \wedge V^{\rm even}$.
\end{definition}

For a pure Sullivan algebra,
we have an explicit construction of the shriek map $\shriekdiag$ and $\shriekinclconst$.
For $\shriekdiag$, see \cite{naito13}\todo{correct?}.
For $\shriekinclconst$, we have the following proposition.

\begin{proposition}
  \label{proposition:shriekInclconstForPure}
  Let $(\wedge V, d)$ be a pure minimal Sullivan algebra.
  Take bases $V^{\mathrm{even}}=\K\{x_1,\dots x_p\}$ and $V^{\mathrm{odd}}=\K\{y_1,\dots y_q\}$.
  Define a $(\wedge V\otimes\wedge \susp{V}, d)$-linear map
  \begin{equation*}
    \shriekinclconst\colon
    (\wedge V\otimes\wedge\susp{V}\otimes\wedge\susp[2]V, d)
    \rightarrow(\wedge V\otimes\wedge\susp{V}, d)
  \end{equation*}
  by $\shriekinclconst(\susp[2]y_1\cdots\susp[2]y_q)=\susp x_1\cdots\susp x_p$
  and $\shriekinclconst(\susp[2]y_{j_1}\cdots\susp[2]y_{j_l})=0$ for $l<q$.
  Then $\shriekinclconst$ defines a non-trivial element in
  $\ext_{\wedge V\otimes\wedge\susp V}
  (\wedge V,\wedge V\otimes\wedge\susp V)$
\end{proposition}
\begin{proof}
  By a straightforward calculation, $\shriekinclconst$ is a cocycle in
  $\hom_{\wedge V\otimes\wedge \susp{V}}
  (\wedge V\otimes\wedge\susp{V}\otimes\wedge\susp[2]V,\wedge V\otimes\wedge\susp{V})$.
  In order to prove the non-triviality,
  we define an ideal
  $I=(x_1,\dots, x_p,y_1,\dots, y_q,\susp y_1,\dots,\susp y_q)\subset\wedge V\otimes\wedge\susp V$.
  By the purity and minimality, we have $d(I)\subset I$.
  Using this ideal, we have the evaluation map of the form
  \begin{align*}
    &\ext_{\wedge V\otimes\wedge\susp V}(\wedge V,\wedge V\otimes\wedge\susp V)
      \otimes \tor_{\wedge V\otimes\wedge\susp V}(\wedge V, \wedge V\otimes\wedge \susp V/I) \\
    &\xrightarrow{\ev} \tor_{\wedge V\otimes\wedge\susp V}(\wedge V\otimes\wedge\susp V, \wedge V\otimes\wedge \susp V/I)
      \xrightarrow{\cong} \wedge\susp V^{\mathrm{even}}.
  \end{align*}
  By this map,
  the element $[\shriekinclconst]\otimes[\susp[2]y_1\cdots\susp[2]y_q\otimes 1]$ is mapped to
  the element $\susp x_1\cdots \susp x_p$, which is obviously non-trivial.
  Hence $[\shriekinclconst]$ is also non-trivial.
\end{proof}

Now, we give proofs of \cref{theorem:braneOperationsOfSphere,theorem:coprodTrivForPure}.

\begin{proof}
  [Proof of \cref{theorem:braneOperationsOfSphere}]
  Using the above descriptions,
  we compute the brane product and coproduct for $M=S^{2n+1}$ and $k=2$.
  In this case,
  we can take $(\wedge V, d) = (\wedge x, 0)$ with $\deg{x}=2n+1$,
  and have $\msphere{1}=(\wedge(x,\susp{x}), 0)$
  and $\mdisk{2}=(\wedge(x,\susp{x},\susp[2]x),d)$ where $dx=d\susp{x}=0$ and $d\susp[2]x=\susp{x}$.
  The computation is straightforward except for the shriek maps $\shriekdiag$ and $\shriekinclconst$.
  The map $\shriekdiag$ is the linear map $\mpath\rightarrow(\wedge x, 0)\tpow2$ over $(\wedge x, 0)\tpow2$
  determined by $\shriekdiag(1)=1\otimes x - x\otimes 1$ and $\shriekdiag((\susp{x})^l)=0$ for $l\geq 1$.
  By \cref{proposition:shriekInclconstForPure}, the map $\shriekinclconst$ is the linear map $\mdisk{k}\rightarrow\msphere{k-1}$ over $\msphere{k-1}$
  determined by $\shriekinclconst(\susp[2]x)=1$ and $\shriekinclconst(1)=0$.

  Then the dual of the brane product $\dprod$ is a linear map
  \begin{equation*}
    \dprod\colon \wedge(x,\susp[2]x)\rightarrow\wedge(x,\susp[2]x)\otimes\wedge(x,\susp[2]x).
  \end{equation*}
  of degree $(1-2n)$
  over $\wedge(x)\otimes\wedge(x)$,
  which is characterized by
  \begin{equation*}
    \dprod(1) = 1\otimes x - x\otimes 1,\ %
    \dprod(\susp[2]x) = (1\otimes x - x\otimes 1)(\susp[2]x\otimes 1 + 1\otimes\susp[2]x).
  \end{equation*}
  Similarly, the dual of the brane coproduct $\dcop$ is a linear map
  \begin{equation*}
    \dcop\colon \wedge(x,\susp[2]x)\otimes\wedge(x,\susp[2]x)\rightarrow\wedge(x,\susp[2]x).
  \end{equation*}
  of degree $(1-2n)$
  over $\wedge(x)\otimes\wedge(x)$,
  which is characterized by
  \begin{equation*}
    \dcop(1) = 0,\ %
    \dcop(\susp[2]x\otimes 1) = -1,\ %
    \dcop(1\otimes\susp[2]x) = 1,\ %
    \dcop(\susp[2]x\otimes\susp[2]x) = -\susp[2]x.
  \end{equation*}

  Dualizing these results,
  we get the brane product and coproduct on the homology,
  which proves \cref{theorem:braneOperationsOfSphere}.
\end{proof}

\begin{proof}
  [Proof of \cref{theorem:coprodTrivForPure}]
  By \cref{proposition:shriekInclconstForPure},
  we have that $\im(\shriekinclconst\otimes\id)$ is contained in the ideal $(\susp x_1,\dots \susp x_p)$,
  which is mapped to zero by the map $\tilde\varepsilon\otimes\id$.
\end{proof}

\section{Proof of the associativity, the commutativity, and the Frobenius compatibility}

\newcommand{\transposeProduct}{\tau_{\mathrm{\times}}}
\newcommand{\transposeConnSum}{\tau_{\mathrm{\#}}}
\label{section:proofOfAssocAndFrob}
In this section,
we give a precise statement and the proof of \cref{theorem:associativeFrobenius}.

First, we give a precise statement of \cref{theorem:associativeFrobenius}.
For simplicity, we omit the statement for $(S,i,j)$-brane product and coproduct,
which is almost the same as that for $(S,T)$-brane product and coproduct.
Let $M$ be a $k$-connected $\K$-Gorenstein space of finite type with $\dim\pi_*(M)\otimes\K < \infty$.
Denote $m=\dim M$.
Then the precise statement of \cref{item:assocProd} is
that the diagrams
\begin{equation}
  \label{equation:assocProdDiagram}
  \xymatrix{
    \cohom{M^{S\#T\#U}} \ar[r]^-{\dprod_{S\#T,U}} \ar[d]^{\dprod_{S,T\#U}} &
    \cohom{M^{S\#T}\times M^U} \ar[d]^{\dprod_{S,T\amalg U}} \\
    \cohom{M^S\times M^{T\#U}} \ar[r]^-{\dprod_{S\amalg T,U}} & \cohom{M^S\times M^T\times M^U}
  }
\end{equation}
and
\begin{equation}
  \label{equation:commProdDiagram}
  \xymatrix{
    \cohom{M^{T\#S}} \ar[r]^-{\dprod_{T,S}} \ar[d]^{\transposeConnSum^*} & \cohom{M^T\times M^S} \ar[d]^{\transposeProduct^*}\\
    \cohom{M^{S\#T}} \ar[r]^-{\dprod_{S,T}} & \cohom{M^S\times M^T}
  }
\end{equation}
commute by the sign $(-1)^m$.\todo{commutes by the sign?}
Here, $\transposeProduct$ and $\transposeConnSum$ are defined as the transposition of $S$ and $T$.
Note that the associativity of the product holds even if the assumption $\dim\pi_*(M)\otimes\K < \infty$ is dropped.

% Assume that % $M$ is $(k-1)$-connected
% $\pi_*(M)\otimes\K = \bigoplus_n\pi_n(M)\otimes\K$ is of finite dimension.
Denote $\bar{m} = \dim\Omega^{k-1} M$.
Then \cref{item:assocCop} states that the diagrams
\begin{equation}
  \label{equation:assocCopDiagram}
  \xymatrix{
    \cohom{M^S\times M^T\times M^U} \ar[r]^-{\dcop_{S\amalg T, U}} \ar[d]^{\dcop_{S,T\amalg U}} &
    \cohom{M^S\times M^{T\#U}} \ar[d]^{\dcop_{S,T\#U}} \\
    \cohom{M^{S\#T}\times M^U} \ar[r]^-{\dcop_{S\#T,U}} & \cohom{M^{S\#T\#U}}
  }
\end{equation}
and
\begin{equation}
  \label{equation:commCopDiagram}
  \xymatrix{
    \cohom{M^{T\times S}} \ar[r]^-{\dcop_{T,S}} \ar[d]^{\transposeConnSum^*} & \cohom{M^T\#M^S} \ar[d]^{\transposeProduct^*}\\
    \cohom{M^{S\times T}} \ar[r]^-{\dcop_{S,T}} & \cohom{M^S\#M^T}
  }
\end{equation}
commute by the sign $(-1)^{\bar{m}}$.
Similarly, \cref{item:Frob} states that the diagram
\begin{equation}
  \label{equation:FrobDiagram}
  \xymatrix{
    \cohom{M^S\times M^{T\#U}} \ar[r]^-{\dcop_{S,T\#U}} \ar[d]^{\dprod_{S\#T,U}} &
    \cohom{M^{S\#T\#U}} \ar[d]^{\dprod_{S\amalg T,U}} \\
    \cohom{M^S\times M^T\times M^U} \ar[r]^-{\dcop_{S,T\amalg U}} & \cohom{M^{S\#T}\times M^U}
  }
\end{equation}
commutes by the sign $(-1)^{m\bar{m}}$.
\todo{the other diagram?}

\newcommand{\lift}[2]{#1_{#2}}

Before proving \cref{theorem:associativeFrobenius},
we give a notation $\lift{g}{\alpha}$ for a shriek map.

\begin{definition}
  Consider a pullback diagram
  \begin{equation*}
    \xymatrix{
      X \ar[r]^g \ar[d]^p & Y \ar[d]^q \\
      A \ar[r]^f & B
    }
  \end{equation*}
  of spaces, where $q$ is a fibration.
  Let $\alpha$ be an element of $\ext^m_{\cochain{B}}(\cochain{A}, \cochain{B})$.
  Assume that the Eilenberg-Moore map
  \begin{equation*}
    \EM\colon \tor^*_{\cochain{B}}(\cochain{A}, \cochain{Y})\xrightarrow{\cong}\cohom{X}
  \end{equation*}
  is an isomorphism (e.g., $B$ is 1-connected and the cohomology of the fiber is of finite type).
  Then we define $\lift{g}{\alpha}$ to be the composition
  \begin{equation*}
    \lift{g}{\alpha}\colon \cohom{X}
    \xleftarrow{\cong} \tor^*_{\cochain{B}}(\cochain{A}, \cochain{Y})
    \xrightarrow{\tor(\alpha, \id)} \tor^{*+m}_{\cochain{B}}(\cochain{B}, \cochain{Y})
    \xrightarrow{\cong} \cohom[*+m]{Y}
  \end{equation*}
\end{definition}

Using this notation,
we can write
the shriek map $\shriek\incl$ as $\lift{\incl}{\shriek\diag}$
for the diagram \cref{equation:STProdDiagram},
and the shriek map $\shriek\comp$ as $\lift{\comp}{\shriek\inclconst}$
for the diagram \cref{equation:STCopDiagram}.

Now we have the following two propositions
as a preparation of the proof of \cref{theorem:associativeFrobenius}.

% \begin{proposition}
%   \label{proposition:naturalityOfShriek}
%   Consider a diagram
%   \begin{equation*}
%     \begin{tikzcd}[row sep=2.5em]
%       X' \arrow[rr,"g'"] \arrow[dr,<-,swap,"\varphi"] \arrow[dd,swap,"p'"] &&
%       Y' \arrow[dd,swap,"q'" near start] \arrow[dr,<-,"\psi"] \\
%       & X \arrow[rr,crossing over,"g" near start] \arrow[dl, "p"] &&
%       Y \arrow[dl, "q"] \\
%       A \arrow[rr,"f" near end] && B,
%     \end{tikzcd}
%   \end{equation*}
%   where $q$ and $q'$ are fibrations
%   and the lower front square and the back square are pullback diagrams.
%   Let $\alpha$ be an element in $\ext^m_{\cochain{B}}(\cochain{A}, \cochain{B})$.
%   Assume that the Eilenberg-Moore maps of two pullback diagrams are isomorphisms.
%   Then the following diagram commutes.
%   \begin{equation*}
%     \xymatrix{
%     \cohom{X'} \ar[r]^-{\lift{g'}{\alpha}} \ar[d]^{\varphi^*} & \cohom[*+m]{Y'} \ar[d]^{\psi^*}\\
%     \cohom{X} \ar[r]^{\lift{g}{\alpha}} & \cohom[*+m]{Y}
%   }
%   \end{equation*}
% \end{proposition}
\begin{proposition}
  \label{proposition:naturalityOfShriek}
  Consider a diagram
  \begin{equation*}
    \begin{tikzcd}[row sep=2.5em]
      X \arrow[rr,"g"] \arrow[dr,swap,"\varphi"] \arrow[dd,swap,""] &&
      Y \arrow[dd,swap,"q" near start] \arrow[dr,"\psi"] \\
      & X' \arrow[rr,crossing over,"g'" near start] &&
      Y' \arrow[dd,"q'"] \\
      A \arrow[rr,"" near end] \arrow[dr,swap,"a"] && B \arrow[dr,swap,"b"] \\
      & A' \arrow[rr,""] \arrow[uu,<-,crossing over,"" near end]&& B',
    \end{tikzcd}
  \end{equation*}
  where $q$ and $q'$ are fibrations
  and the front and back squares are pullback diagrams.
  Let $\alpha \in \ext^m_{\cochain{B}}(\cochain{A}, \cochain{B})$
  and $\alpha' \in \ext^m_{\cochain{B'}}(\cochain{A'}, \cochain{B'})$.
  Assume that the elements $\alpha$ and $\alpha'$
  are mapped to the same element in $\ext^m_{\cochain{B'}}(\cochain{A'}, \cochain{B})$
  by the morphisms induced by $a$ and $b$,
  and that the Eilenberg-Moore maps of two pullback diagrams are isomorphisms.
  Then the following diagram commutes.
  \begin{equation*}
    \xymatrix{
      \cohom{X'} \ar[r]^-{\lift{g'}{\alpha'}} \ar[d]^{\varphi^*} & \cohom[*+m]{Y'} \ar[d]^{\psi^*}\\
      \cohom{X} \ar[r]^{\lift{g}{\alpha}} & \cohom[*+m]{Y}
    }
  \end{equation*}
\end{proposition}

\begin{proposition}
  \label{proposition:functorialityOfShriek}
  Consider a diagram
  \begin{equation*}
    \xymatrix{
      X \ar[r]^{\tilde{f}} \ar[d]^p & Y \ar[r]^{\tilde{g}} \ar[d]^q & Z \ar[d]^r \\
      A \ar[r]^f & B \ar[r]^g & C,
    }
  \end{equation*}
  where the two squares are pullback diagrams.
  Let $\alpha$ be an element of $\ext^m_{\cochain{B}}(\cochain{A}, \cochain{B})$
  and $\beta$ an element of $\ext^n_{\cochain{C}}(\cochain{B}, \cochain{C})$.
  Assume that the Eilenberg-Moore maps are isomorphisms for two pullback diagrams.
  Then we have
  \begin{equation*}
    \lift{(\tilde{g}\circ\tilde{f})}{\beta\circ (g_*\alpha)}
    = \lift{\tilde{g}}{\beta} \circ\lift{\tilde{f}}{\alpha},
  \end{equation*}
  where
  $g_*\colon \ext^m_{\cochain{B}}(\cochain{A}, \cochain{B}) \rightarrow \ext^m_{\cochain{C}}(\cochain{A}, \cochain{B})$
  is the morphism induced by the map $g\colon B\rightarrow C$.
\end{proposition}

These propositions can be proved by straightforward arguments.

\begin{proof}
  [Proof of \cref{theorem:associativeFrobenius}]
  First, we give a proof for \cref{item:Frob}.
  Note that the associativity in \cref{item:assocProd} and \cref{item:assocCop} can be proved similarly.

  Consider the following diagram.
  % \begin{equation*}
  %   \xymatrix{
  %   \cohom{M^{S\#T\#U}} \ar[r] \ar[d] &
  %   \cohom{M^{S\#T}\times_MM^U} \ar[r] \ar[d] &
  %   \cohom{M^{S\#T}\times M^U} \ar[d] \\
  %     %
  %   \cohom{M^S\times_MM^{T\#U}} \ar[r] \ar[d] &
  %   \cohom{M^S\times_MM^T\times_MM^U} \ar[r] \ar[d] &
  %   \cohom{M^S\times_MM^T\times M^U} \ar[d]\\
  %     %
  %   \cohom{M^S\times M^{T\#U}} \ar[r] &
  %   \cohom{M^S\times M^T\times_MM^U} \ar[r] &
  %   \cohom{M^S\times M^T\times M^U}
  % }
  % \end{equation*}
  % \begin{equation*}
  %   \xymatrix{
  %   \cohom{M^S\times M^T\times M^U} \ar[r] \ar[d] &
  %   \cohom{M^S\times M^T\times_MM^U} \ar[r] \ar[d] &
  %   \cohom{M^S\times M^{T\#U}} \ar[d] \\
  %     %
  %   \cohom{M^S\times_MM^T\times M^U} \ar[r] \ar[d] &
  %   \cohom{M^S\times_MM^T\times_MM^U} \ar[r] \ar[d] &
  %   \cohom{M^S\times_MM^{T\#U}} \ar[d]\\
  %     %
  %   \cohom{M^{S\#T}\times M^U} \ar[r] &
  %   \cohom{M^{S\#T}\times_MM^U} \ar[r] &
  %   \cohom{M^{S\#T\#U}}
  % }
  % \end{equation*}
  \begin{equation*}
    \xymatrix{
      \cohom{M^S\times M^{T\#U}} \ar[r]^-{\incl^*} \ar[d]^-{\comp^*} &
      \cohom{M^S\times_MM^{T\#U}} \ar[r]^-{\lift\comp{\shriek\inclconst}} \ar[d]^-{\comp^*} &
      \cohom{M^{S\#T\#U}} \ar[d]^-{\comp^*} \\
      \cohom{M^S\times M^T\times_MM^U} \ar[r]^-{\incl^*} \ar[d]^-{\lift\incl{\shriek\diag}} &
      \cohom{M^S\times_MM^T\times_MM^U} \ar[r]^-{\lift\comp{\shriek{(\inclconst\times\id)}}} \ar[d]^-{\lift\incl{\shriek{(\id\times\diag)}}} &
      \cohom{M^{S\#T}\times_MM^U} \ar[d]^-{\lift\incl{\shriek{(\id\times\diag)}}} \\
      \cohom{M^S\times M^T\times M^U} \ar[r]^-{\incl^*} &
      \cohom{M^S\times_MM^T\times M^U} \ar[r]^-{\lift\comp{\shriek{(\inclconst\times\id)}}} &
      \cohom{M^{S\#T}\times M^U}
    }
  \end{equation*}
  Note that the boundary of the whole square is the same as the diagram \cref{equation:FrobDiagram}.
  The upper left square is commutative by the functoriality of the cohomology and
  so are the upper right and lower left squares by \cref{proposition:naturalityOfShriek}.
  Next, we consider the lower right square.
  Applying \cref{proposition:functorialityOfShriek} to the diagram
  \begin{equation*}
    \xymatrix{
      M^S\times_MM^T\times_MM^U \ar[r]^-{\comp} \ar[d] & M^{S\#T}\times_MM^U \ar[r]^-{\incl} \ar[d] & M^{S\#T}\times M^U \ar[d]\\
      M\times M \ar[r]^-{\inclconst\times\id} & \spheresp[k-1]{M}\times M \ar[r]^-{\id\times\diag} & \spheresp[k-1]{M}\times M^2,
    }
  \end{equation*}
  we have
  \begin{equation*}
    \lift\incl{\shriek{(\id\times\diag)}}
    \circ \lift\comp{\shriek{(\inclconst\times\id)}}
    = \lift{(\incl\circ\comp)}
    {\shriek{(\id\times\diag)} \circ ((\id\times\diag)_*\shriek{(\inclconst\times\id)})}.
  \end{equation*}
  Using appropriate semifree resolutions,
  we have a representation
  \begin{align*}
    \shriek{(\id\times\diag)} \circ ((\id\times\diag)_*\shriek{(\inclconst\times\id)})
    & = [\id\otimes\shriekdiag] \circ [\shriekinclconst\otimes\id] \\
    & = [(-1)^{m\bar{m}}\shriekinclconst\otimes\shriekdiag]
    % & = (-1)^{m\bar{m}}\shriek\inclconst\otimes\shriek\diag \\
  \end{align*}
  as a chain map.\todo{write more detailed proof}
  Here,
  $[\shriekdiag]=\shriek\diag\in\ext_{\cochain{M^2}}^m(\cochain{M}, \cochain{M^2})$
  and
  $[\shriekinclconst]=\shriek\inclconst\in\ext^{\bar{m}}_{\cochain{\spheresp[k-1]{M}}}(\cochain{M}, \cochain{\spheresp[k-1]{M}})$
  are representations as cochains.
  Similarly, we compute the other composition to be
  \begin{equation*}
    \lift\comp{\shriek{(\inclconst\times\id)}}
    \circ \lift\incl{\shriek{(\id\times\diag)}}
    = \lift{(\comp\circ\incl)}
    {\shriek{(\inclconst\times\id)} \circ ((\inclconst\times\id)_*)\shriek{(\id\times\diag)}}
  \end{equation*}
  with
  \begin{equation*}
    \shriek{(\inclconst\times\id)} \circ ((\inclconst\times\id)_*)\shriek{(\id\times\diag)}
    = [\shriekinclconst\otimes\shriekdiag].
  \end{equation*}
  This proves the commutativity by the sign $(-1)^{m\bar{m}}$ of the lower right square.
  % The lower right square is commutatie by the sign $(-1)^{m\bar{m}}$ by \cref{proposition:functorialityOfShriek}.

  Next, we prove the commutativity of the coproduct in \cref{item:assocCop}.
  This follows from the commutativity of the diagram
  \begin{equation}
    \label{equation:proofOfCommutativityDiagram}
    \xymatrix{
      \cohom{M^T\times M^S} \ar[r]^{\incl^*} \ar[d]^{\transposeProduct^*}
      & \cohom{M^T\times_MM^S} \ar[r]^-{\shriek{\comp}} \ar[d]^{\transposeProduct^*}
      & \cohom{M^{T\#S}} \ar[d]^{\transposeConnSum^*}\\
      \cohom{M^S\times M^T} \ar[r]^{\incl^*}
      & \cohom{M^S\times_MM^T} \ar[r]^-{\shriek{\comp}}
      & \cohom{M^{S\#T}}.
    }
  \end{equation}
  The commutativity of the left square is obvious.
  If one can apply \cref{proposition:naturalityOfShriek} to the diagram
  \cref{equation:cubeDiagramForCommutativity},
  we obtain the commutativity of the right square of \cref{equation:proofOfCommutativityDiagram}.
  \begin{equation}
    \label{equation:cubeDiagramForCommutativity}
    \begin{tikzcd}[row sep=2.5em]
      M^S\times_MM^T \arrow[rr,"\comp"] \arrow[dr,swap,"\transposeProduct"] \arrow[dd,swap,""] &&
      M^{S\#T} \arrow[dd,swap,"\res" near start] \arrow[dr,"\transposeConnSum"] \\
      & M^T\times_MM^S \arrow[rr,crossing over,"\comp" near start] &&
      M^{T\#S} \arrow[dd,"\res"] \\
      M \arrow[rr,"\inclconst" near end] \arrow[dr,swap,"\id"] && \spheresp[k-1]{M} \arrow[dr,swap,"\orirev"] \\
      & M \arrow[rr,"\inclconst"] \arrow[uu,<-,crossing over,"" near end]&& \spheresp[k-1]{M}
    \end{tikzcd}
  \end{equation}
  In order to apply \cref{proposition:naturalityOfShriek},
  it suffices to prove the equation
  \begin{equation}
    \label{equation:commutativityOfInclconstShriek}
    \ext_{\orirev^*}(\id,\orirev^*)(\shriek{\inclconst})=(-1)^{\bar{m}}\shriek{\inclconst}
  \end{equation}
  in $\ext_{\cochain{\spheresp[k-1]{M}}}(\cochain{M},\cochain{\spheresp[k-1]{M}})$.
  Since $\ext^{\bar{m}}_{\cochain{\spheresp[k-1]{M}}}(\cochain{M}, \cochain{\spheresp[k-1]{M}}) \cong \K$ and
  $\ext_{\orirev^*}(\id,\orirev^*) \circ \ext_{\orirev^*}(\id,\orirev^*) = \id$,
  we have \cref{equation:commutativityOfInclconstShriek} up to sign.
  % \begin{equation}
  %   \label{equation:commutativityOfInclconstShriekUpToSign}
  %   \ext_{\orirev^*}(\id,\orirev^*)(\shriek{\inclconst})=(-1)^{\bar{l}}\shriek{\inclconst}
  % \end{equation}
  % for some $\bar{l}\in\Z$.
  In \cref{section:proofOfExtAlgebraic},
  we will determine the sign to be $(-1)^{\bar{m}}$.

  Similarly, in order to prove the commutativity of the product in \cref{item:assocProd},
  we need to prove the equation
  \begin{equation}
    \label{equation:commutativityOfDeltaShriek}
    \ext_{\orirev^*}(\id,\orirev^*)(\shriek{\diag})=(-1)^m\shriek{\diag}
  \end{equation}
  in $\ext_{\cochain{M^2}}(\cochain{M},\cochain{M^2})$.
  As above, we have \cref{equation:commutativityOfDeltaShriek} up to sign.
  % \begin{equation}
  %   \label{equation:commutativityOfDeltaShriekUpToSign}
  %   \ext_{\orirev^*}(\id,\orirev^*)(\shriek{\diag})=(-1)^{l}\shriek{\diag}
  % \end{equation}
  % for some $l\in\Z$.
  The sign is determined to be $(-1)^m$ in \cref{section:determineSign}.

  The same proofs can be applied for $(S,i,j)$-brane product and coproduct.
\end{proof}

\section{Proof of \cref{equation:commutativityOfDeltaShriek}}
\label{section:determineSign}
\newcommand{\explShriek}{f}
In this section, we will prove \cref{equation:commutativityOfDeltaShriek}, determining the sign.
Here, we need the explicit description of $\shriek{\diag}$ in \cite{wakatsuki16}.

Let $M$ be a $1$-connected space with $\dim\pi_*(M)\otimes\K < \infty$.
By \cite[Theorem 1.6]{wakatsuki16}, we have a Sullivan model $(\wedge V, d)$ of $M$ which is semi-pure,
i.e., $d(I_V)\subset I_V$, where $I_V$ is the ideal generated by $V^{\mathrm{even}}$.
Let $\varepsilon\colon (\wedge V, d)\rightarrow \K$ be the augmentation map
and $\pr\colon (\wedge V, d)\rightarrow(\wedge V/I_V, d)$ the quotient map.
Take bases $V^{\mathrm{even}}=\K\{x_1,\dots x_p\}$ and $V^{\mathrm{odd}}=\K\{y_1,\dots y_q\}$.
Recall the relative Sullivan algebra $\mpath=(\wedge V\tpow2\otimes\wedge \susp{V}, d)$
over $(\wedge V, d)\tpow2$ from \cref{section:computeExample}.
% by the formula
% \begin{equation*}
%   d(\susp v)=1\otimes v - v\otimes 1 - \sum_{i=1}^\infty\frac{(sd)^i}{i!}(v\otimes 1)
% \end{equation*}
% inductively
% (see \cite[Section 15 (c)]{felix-halperin-thomas01} or \cite[Appendix A]{wakatsuki16} for details\todo{wakatsuki16?,inductively?}).
Note that the relative Sullivan algebra $(\wedge V\tpow2\otimes\wedge \susp{V}, d)$ is a relative Sullivan model of
the multiplication map $(\wedge V, d)^{\otimes 2}\rightarrow (\wedge V, d)$,
Hence, using this as a semifree resolution, we have
$\ext_{\wedge V^\otimes 2}(\wedge V, \wedge V^{\otimes 2})
= \cohom{\hom_{\wedge V^\otimes 2}(\wedge V^{\otimes 2}\otimes\wedge\susp{V}, \wedge V^{\otimes 2})}$.
By \cite[Corollary 5.5]{wakatsuki16}, we have a cocycle
$\explShriek \in \hom_{\wedge V^\otimes 2}(\wedge V^{\otimes 2}\otimes\wedge\susp{V}, \wedge V^{\otimes 2})$
satisfying $\explShriek(\susp{x_1}\cdots\susp{x_p}) = \prod_{j=1}^{j=q}(1\otimes y_j - y_j\otimes 1) + u$
for some $u \in (y_1\otimes y_1, \ldots, y_q\otimes y_q)$.
Consider the evaluation map
\begin{align*}
  \ev\colon \ext_{\wedge V\tpow{2}}(\wedge V, \wedge V\tpow{2})
  \otimes \tor_{\wedge V\tpow{2}}(\wedge V, \wedge V / I_V)
  &\rightarrow \tor_{\wedge V\tpow{2}}(\wedge V\tpow{2}, \wedge V / I_V)\\
  &\xrightarrow{\cong} \cohom{\wedge V / I_V},
\end{align*}
where
$(\wedge V, d)\tpow{2}$, $(\wedge V, d)$, and $(\wedge V / I_V, d)$
are $(\wedge V, d)\tpow{2}$-module via
$\id$, $\varepsilon\cdot\id$, and $\pr\circ(\varepsilon\cdot\id)$, respectively.
Here, we use $(\wedge V\tpow2\otimes\wedge\susp{V},d)$ as a semifree resolution of $(\wedge V, d)$.
Then, we have
\begin{equation*}
  \ev([\explShriek]\otimes[\susp{x_1}\cdots\susp{x_p}]) = [y_1\cdots y_q] \neq 0,
\end{equation*}
and hence $[\explShriek]\neq 0$ in $\ext_{\wedge V\tpow{2}}(\wedge V, \wedge V\tpow{2})$.
\newcommand{\transposeSullivan}{t}
\newcommand{\transposeRelative}{\tilde{\transposeSullivan}}
Thus, it is enough to calculate $\ext_{\transposeSullivan}(\id,\transposeSullivan)([\explShriek])$
to determine the sign in \cref{equation:commutativityOfDeltaShriek},
where $\transposeSullivan\colon (\wedge V, d)\tpow2 \rightarrow (\wedge V, d)$
is the dga homomorphism defined by
$\transposeSullivan(v\otimes1)=1\otimes v$ and $\transposeSullivan(1\otimes v)=v\otimes 1$.

\begin{proof}
  [Proof of \cref{equation:commutativityOfDeltaShriek}]
  By definition, $\ext_{\transposeSullivan}(\id,\transposeSullivan)$ is induced by the map
  \begin{equation*}
    \hom_\transposeSullivan(\transposeRelative, \transposeSullivan)\colon
    \hom_{\wedge V\tpow2}(\wedge V\tpow2\otimes\wedge\susp{V}, \wedge V\tpow2)
    \rightarrow \hom_{\wedge V\tpow2}(\wedge V\tpow2\otimes\wedge\susp{V}, \wedge V\tpow2),
  \end{equation*}
  where $\transposeRelative$ is the dga automorphism defined by
  $\transposeRelative|_{\wedge V\tpow2}=\transposeSullivan$ and
  $\transposeRelative(\susp{v}) = -\susp{v}$.
  Since $\transposeRelative(\susp{x_1}\cdots\susp{x_p})=(-1)^p\susp{x_1}\cdots\susp{x_p}$
  and $\transposeSullivan(\prod_{j=1}^{j=q}(1\otimes y_j - y_j\otimes 1))=(-1)^q\prod_{j=1}^{j=q}(1\otimes y_j - y_j\otimes 1)$,
  we have
  \begin{equation*}
    \ev([\hom_\transposeSullivan(\transposeRelative, \transposeSullivan)(\explShriek)]
    \otimes[\susp{x_1}\cdots\susp{x_p}])
    = \ev([\transposeSullivan\circ\explShriek\circ\transposeRelative]
    \otimes[\susp{x_1}\cdots\susp{x_p}])
    = (-1)^{p+q}[y_1\cdots y_q].
  \end{equation*}
  Since the parity of $p+q$ is the same as that of the dimension of $(\wedge V, d)$ as a Gorenstein algebra,
  the sign in \cref{equation:commutativityOfDeltaShriek} is proved to be $(-1)^{m}$.
\end{proof}

\section{Proof of \cref{equation:commutativityOfInclconstShriek}}
\label{section:proofOfExtAlgebraic}
\newcommand{\resol}{\eta}
In this section, we give the proof of \cref{equation:commutativityOfInclconstShriek},
using the spectral sequence constructed in the proof of \cref{theorem:extAlgebraic}.
Although the key idea of the proof of \cref{theorem:extAlgebraic} is
the same as \cref{theorem:ExtDiagonal} due to F\'elix and Thomas,
we give the proof here for the convenience of the reader.
\todo{??}

\begin{proof}
  [Proof of \cref{theorem:extAlgebraic}]
  Take a $(A\otimes B, d)$-semifree resolution $\resol\colon(P,d)\xrightarrow{\qis}(A,d)$.
  Define $(C,d)=(\hom_{A\otimes B}(P,A\otimes B),d)$.
  Then $\ext_{A\otimes B}(A,A\otimes B) = \cohom{C,d}$.
  We fix a non-negative integer $N$, and define a complex
  $(C_N,d) = (\hom_{A\otimes B}(P,(A/A^{>n})\otimes B),d)$.
  We will compute the cohomology of $(C_N,d)$.
  Define a filtration $\{F^pC_N\}_{p\geq 0}$ on $(C_N,d)$ by
  $F^pC_N = \hom_{A\otimes B}(P,(A/A^{>n})^{\geq p}\otimes B)$.
  Then we obtain a spectral sequence $\{E^{p,q}_r\}_{r\geq 0}$
  converging to $\cohom{C_N, d}$.
  \begin{claim}
    \label{lemma:E2termOfSS}
    \begin{equation*}
      E^{p,q}_2 =
      \begin{cases}
        \cohom[p]{A/A^{>N}} & \mbox{(if $q=m$)}\\
        0 & \mbox{(if $q\neq m$)}
      \end{cases}
    \end{equation*}
  \end{claim}
  \begin{proof}
    [Proof of \cref{lemma:E2termOfSS}]
    We may assume $p\leq N$.
    Then we have an isomorphism of complexes
    \begin{equation*}
      (A^{\geq p}/A^{\geq p+1},0)\otimes (\hom_B(B\otimes_{A\otimes B}P,B),d)
      \xrightarrow{\cong} (E^p_0, d_0),
    \end{equation*}
    hence
    \begin{equation*}
      (A^{\geq p}/A^{\geq p+1})\otimes \cohom{\hom_B(B\otimes_{A\otimes B}P,B),d}
      \xrightarrow{\cong} E^p_1.
    \end{equation*}
    Define
    \begin{equation*}
      \bar{\resol}\colon (B,\bar{d})\otimes_{A\otimes B}(P,d)
      \xrightarrow{1\otimes\resol}(B,\bar{d})\otimes_{A\otimes B}(A,d)
      \cong \K.
    \end{equation*}
    Note that the last isomorphism follows from the assumption \cref{item:assumpResId}.
    Then, since $\resol$ is a quasi-isomorphism, so is $\bar{\resol}$.
    Hence we have
    \begin{equation*}
      \cohom[q]{\hom_B(B\otimes_{A\otimes B}P,B),d}
      \cong \ext^q_B(\K, B) \cong
      \begin{cases}
        \K & \mbox{(if $q = m$)}\\
        0 & \mbox{(if $q\neq m$)}
      \end{cases}
    \end{equation*}
    by the assumption \cref{item:assumpGorenstein}.

    Hence we have
    \begin{align*}
      E^{p,q}_1 &\cong (A^{\geq p}/A^{\geq p+1}) \otimes \cohom[q]{\hom_B(B\otimes_{A\otimes B}P, B),d}\\
                &\cong A^p\otimes\ext^q_B(\K, B).
    \end{align*}
    Moreover, using the assumption \cref{item:assump1conn} and the above isomorphisms,
    we can compute the differential $d_1$ and have an isomorphism of complexes
    \begin{equation}
      \label{equation:computationOfE1}
      (E^{*,q}_1,d_1) \cong (A^*,d) \otimes \ext^q_B(\K, B).
    \end{equation}
    This proves \cref{lemma:E2termOfSS}.
  \end{proof}

  Now we return to the proof of \cref{theorem:extAlgebraic}.
  We will recover $\cohom{C}$ from $\cohom{C_N}$ taking a limit.
  Since $\limone_NC_N=0$, we have an exact sequence
  \begin{equation*}
    0\rightarrow \limone_N\cohom{C_N}
    \rightarrow \cohom{\lim_NC_N}
    \rightarrow \cohom{\lim_N\cohom{C_N}}
    \rightarrow 0.
  \end{equation*}
  By \cref{lemma:E2termOfSS},
  the sequence $\{\cohom{C_N}\}_N$ satisfies the (degree-wise) Mittag-Leffler condition,
  and hence $\limone_N\cohom{C_N}=0$.
  Thus, we have
  \begin{equation*}
    \cohom[l]{C}
    \cong \cohom[l]{\lim_NC_N}
    \cong \lim_N\cohom[l]{C_N}
    \cong \cohom[l-m]{A}.
  \end{equation*}
  This proves \cref{theorem:extAlgebraic}.
\end{proof}

\newcommand{\orirevSullivan}{t}
\newcommand{\orirevRelative}{{\tilde{\orirevSullivan}}}
\newcommand{\orirevFiber}{{\bar{\orirevSullivan}}}
\newcommand{\orirevFiberResol}{{\hat{\orirevSullivan}}}
Next, using the above spectral sequence,
we determine the sign in \cref{equation:commutativityOfInclconstShriek}.

\begin{proof}
  [Proof of \cref{equation:commutativityOfInclconstShriek}]
  If $k=1$, \cref{equation:commutativityOfInclconstShriek} is the same as \cref{equation:commutativityOfDeltaShriek},
  which was proved in \cref{section:determineSign}.
  Hence we assume $k\geq 2$.
  As in \cref{section:proofOfAssocAndFrob},
  let $M$ be a $k$-connected $\K$-Gorenstein space of finite type with $\dim\pi_*(M)\otimes\K < \infty$,
  and $(\wedge V, d)$ its minimal Sullivan model.
  Using the Sullivan models constructed in \cref{section:constructModel},
  we have that the automorphism $\ext_{\orirev^*}(\id,\orirev^*)$
  on $\ext_{\cochain{\spheresp[k-1]{M}}}(\cochain{M},\cochain{\spheresp[k-1]{M}})$
  is induced by the automorphism $\hom_\orirevSullivan(\orirevRelative, \orirevSullivan)$
  on
  $\hom_{\wedge V\otimes\wedge\susp[k-1]{V}}
  (\wedge V\otimes\wedge\susp[k-1]{V}\otimes\wedge\susp[k]{V}, \wedge V\otimes\wedge\susp[k-1]{V})$,
  where $\orirevSullivan$ and $\orirevRelative$ are the dga automorphisms
  on $(\wedge V\otimes\wedge\susp[k-1]{V}, d)$
  and $(\wedge V\otimes\wedge\susp[k-1]{V}\otimes\wedge\susp[k]{V}, d)$, respectively,
  defined by
  \begin{align*}
    &\orirevSullivan(v) = v,\ \orirevSullivan(\susp[k-1]{v}) = -\susp[k-1]{v},\\
    &\orirevRelative(v) = v,\ \orirevRelative(\susp[k-1]{v}) = -\susp[k-1]{v},\ \mathrm{and}\ %
      \orirevRelative(\susp[k]{v}) = -\susp[k]{v}.
  \end{align*}

  Now, consider the spectral sequence $\{E^{p,q}_r\}$ in the proof of \cref{theorem:extAlgebraic}
  by taking $(A\otimes B, d) = (\wedge V \otimes \wedge \susp[k-1]{V}, d)$
  and $(P,d) = (\wedge V\otimes\wedge\susp[k-1]{V}\otimes\wedge\susp[k]{V}, d)$.
  % Consider the complexes $C_N$ and $F^pC_N$ defined in the proof of \cref{theorem:extAlgebraic}.
  Since $k\geq 2$,
  $\hom_\orirevSullivan(\orirevRelative, \orirevSullivan)$ induces
  automorphisms on the complexes $C_N$ and $F^pC_N$,
  and hence on the spectral sequence $\{E^{p,q}_r\}$.
  By the isomorphism \cref{equation:computationOfE1},
  we have
  \begin{equation*}
    E^{p,q}_2\cong \cohom[p]{A}\otimes\ext^q_{\wedge \susp[k-1]{V}}(\K, \wedge \susp[k-1]{V}),
  \end{equation*}
  and that the automorphism induced on $E_2$ is the same as $\id\otimes\ext_\orirevFiber(\id, \orirevFiber)$,
  where $\orirevFiber$ is defined by $\orirevFiber(\susp[k-1]{v})=-\susp[k-1]{v}$ for $v\in V$.
  Since the differential is zero on $\wedge \susp[k-1]{V}$,
  we have an isomorphism
  \begin{equation*}
    \ext^*_{\wedge \susp[k-1]{V}}(\K, \wedge \susp[k-1]{V})
    \cong \bigotimes_i\ext^*_{\wedge \susp[k-1]{v_i}}(\K, \wedge \susp[k-1]{v_i})
  \end{equation*}
  where $\{v_1,\ldots,v_l\}$ is a basis of $V$.
  Using this isomorphism, we can identify
  \begin{equation*}
    \ext_\orirevFiber(\id, \orirevFiber) = \bigotimes_i\ext_{\orirevFiber_i}(\id, \orirevFiber_i),
  \end{equation*}
  where $\orirevFiber_i$ is defined by $\orirevFiber_i(\susp[k-1]{v_i})=-\susp[k-1]{v_i}$.

  Since $(-1)^{\dim V}=(-1)^{\bar{m}}$,
  it suffices to show $\ext_{\orirevFiber_i}(\id, \orirevFiber_i)=-1$.
  Taking a resolution, we have
  \begin{align*}
    &\ext^*_{\wedge \susp[k-1]{v_i}}(\K, \wedge \susp[k-1]{v_i})
      = \cohom{\hom_{\wedge \susp[k-1]{v_i}}
      (\wedge \susp[k-1]{v_i}\otimes\wedge\susp[k]{v_i}, \wedge \susp[k-1]{v_i})}\\
    &\ext_{\orirevFiber_i}(\id, \orirevFiber_i)=\cohom{\hom_{\orirevFiber_i}(\orirevFiberResol_i, \orirevFiber_i)},
  \end{align*}
  where the differential $d$ on $\wedge \susp[k-1]{v_i}\otimes\wedge\susp[k]{v_i}$
  is defined by $d(\susp[k-1]{v_i})=0$ and $d(\susp[k]{v_i})=\susp[k-1]{v_i}$,
  and the dga homomorphism $\orirevFiberResol_i$ is defined by
  $\orirevFiberResol_i(\susp[k-1]{v_i})=-\susp[k-1]{v_i}$ and $\orirevFiberResol_i(\susp[k]{v_i})=-\susp[k]{v_i}$.
  \newcommand{\extgen}{f}
  Using this resolution, we have the generator $[\extgen]$ of
  $\cohom{\hom_{\wedge \susp[k-1]{v_i}}(\wedge \susp[k-1]{v_i}\otimes\wedge\susp[k]{v_i}, \wedge \susp[k-1]{v_i})}\cong\K$
  as follows:
  \begin{itemize}
    \item If $\deg{\susp[k-1]{v_i}}$ is odd,
      define $\extgen$ by $\extgen(1)=\susp[k-1]{v_i}$ and $\extgen((\susp[k]{v_i})^l)=0$ for $l\geq 1$.
    \item If $\deg{\susp[k-1]{v_i}}$ is even,
      define $\extgen$ by $\extgen(1)=0$ and $\extgen((\susp[k]{v_i}))=1$.
  \end{itemize}
  In both cases, we have
  $\hom_{\orirevFiber_i}(\orirevFiberResol_i, \orirevFiber_i)(\extgen)
  =\orirevFiber_i\circ\extgen\circ\orirevFiberResol_i
  =-\extgen$.
  This proves $\ext_{\orirevFiber_i}(\id, \orirevFiber_i)=-1$
  and completes the determination of the sign in \cref{equation:commutativityOfInclconstShriek}.
  % Since the differential is zero on $\wedge \susp[k-1]{V}$,
  % it is easy to prove $\ext_\orirevFiber(\id, \orirevFiber) = (-1)^{\bar{m}}$.
  % In fact, we can compute the module by
  % \begin{equation*}
  %   \ext^*_{\wedge \susp[k-1]{V}}(\K, \wedge \susp[k-1]{V})
  %   = \cohom{\hom_{\wedge \susp[k-1]{V}}(\wedge \susp[k-1]{V}\otimes\wedge\susp[k]{V}, \wedge \susp[k-1]{V})},
  % \end{equation*}
  % where the differential $d$ on the dga $(\wedge \susp[k-1]{V}\otimes\wedge\susp[k]{V},d)$ is defined by
  % $d(\susp[k-1]v)=0$ and $d(\susp[k]v)=\susp[k-1]v$ for $v \in V$.
  % By this isomorphism, $\ext_\orirevFiber(\id, \orirevFiber)$ is induced by
  % $\hom_\orirevFiber(\orirevFiberResol,\orirevFiber)$,
  % where $\orirevFiberResol$ is defined by
  % $\orirevFiberResol(\susp[k-1]v)=-\susp[k-1]v$ and $\orirevFiberResol(\susp[k]v)=-\susp[k]v$.
  % Using this, we have \todo{not finished}.
\end{proof}

\section*{Acknowledgment}
I would like to express my gratitude to Katsuhiko Kuribayashi and Takahito Naito for productive discussions and valuable suggestions.
Furthermore, I would like to thank my supervisor Nariya Kawazumi for the enormous support and comments.
This work was supported by JSPS KAKENHI Grant Number 16J06349 and the Program for Leading Graduate School, MEXT, Japan.

% \bibliographystyle{halphaabbrv}
% \bibliography{references}

\end{document}